\documentclass{amsart}

\usepackage[english]{babel}

\usepackage{a4wide}
\usepackage{amsmath}
\usepackage{graphicx}
\usepackage[colorlinks=true, allcolors=blue]{hyperref}

\def\A{\mathcal{A}}
\def\Q{\mathcal{Q}}
\def\dt{\textrm{d}t}
\def\d{\textrm{d}}
\DeclareMathOperator{\Id}{I}
\newcommand{\norm}[1]{\left\lVert#1\right\rVert}

\newcommand{\R}{\mathbb{R}}

\theoremstyle{plain}
\newtheorem{theorem}{Theorem}[section]

\newtheorem{lemma}{Lemma}[section]
\newtheorem{proposition}{Proposition}[section]
\newtheorem{corollary}{Corollary}[section]

\theoremstyle{definition}
\newtheorem{definition}{Definition}[section]

\theoremstyle{remark}
\newtheorem{remark}{Remark}[section]

\numberwithin{equation}{section}

\title{A complete theory of local minimizers for periodic variational problems}
\author{Marco Fenucci}

\address{Marco Fenucci \newline
Planetary Defence Office / NEO Coordination Centre, ESA ESRIN, Largo Galileo Galilei 1, Frascati, 00044, RM, Italy}
\email{marco.fenucci@ext.esa.int}

\begin{document}

\begin{abstract}
We study the periodic problem in the Calculus of Variations and give a complete characterization of the positive definiteness of the second variation, extending the Reid Roundabout Theorem to periodic boundary conditions. Two formulations are established: one in terms of conjugate and semi-coupled points, and one in terms of the existence of a symmetric solution of the associated Riccati equation satisfying a boundary inequality. We show by an explicit example that the latter condition, while sufficient, is not necessary except in dimension one. As applications, we obtain sufficient conditions for weak and strong local minimizers of periodic solutions of the Euler-Lagrange equation.

\textbf{MSC 2020 classification codes:} 49K15 -- 34B05 -- 49N10 
\end{abstract}
\maketitle

\section{Introduction}
The fixed end-point problem (or \emph{problem of Bolza}) consists in minimizing an integral function over all the $C^1$ curves $u:[t_1, t_2] \to \R^n$ such that
\begin{equation}
    \begin{cases}
        u(t_1) = u_1, \\
        u(t_2) = u_2,
    \end{cases}
\end{equation}
where $u_1,u_2 \in \R^n$ are two points in the space. This is the classical problem of Calculus of Variations, and its study dates back to the 17th century. In fact, the first problem of this type to be studied is that of the brachistochrone curve, which solves the following problem: given two points $A$ and $B$ in a vertical plane, and a point mass $m$ moving under a constant gravitational field, which is the curve traced by the point which starts from $A$ and reaches $B$ in the shortest time?
Galileo Galilei already addressed a version of this problem in 1638, noticing that a point mass moving on the arc of a circle is faster than one moving along the corresponding straight line \cite{galileo_1638}, although he incorrectly conjectured the circular arc itself to be the solution.
The problem in its modern form was posed by Johann Bernoulli in 1696, as a challenge to the mathematicians of Europe. Five solutions were submitted by Newton, Leibniz, Jacob Bernoulli, Johann Bernoulli himself, and de l'H\^opital, all showing that the solution is given by an arc of a cycloid. 
A general and systematic treatment followed with Euler and Lagrange derivation of the first-order necessary condition, now known as the Euler-Lagrange equation, and was completed over the following century by Legendre, Jacobi, and Weierstrass, who established necessary and sufficient conditions for local minimality based on the second variation of the functional \cite{pars_1962, dacorogna_2004}. The classical theory reached its modern form in the work of Reid \cite{Reid1980}, who reformulated the Legendre-Jacobi theory in terms of the existence of a symmetric, non-singular solution of an associated matrix Riccati differential equation on the whole interval $[t_1,t_2]$. The resulting equivalence is now known as the Reid Roundabout Theorem.

A natural generalization concerns boundary conditions that couple the values of the curve at both endpoints, rather than fixing them independently -- the \emph{periodic} condition $u(t_1)=u(t_2)$ being one of the simplest example.
The minimization problem with periodic boundary conditions has seen a significant interest in the field of Celestial Mechanics in the last 30 years, since the numerical discovery of the figure-eight solution of the three-body problem \cite{moore_1993}. The figure-eight solution and many other symmetric periodic solutions of the $n$-body problem were later found as minimizers of the Lagrangian action functional, see e.g. \cite{chenciner-montgomery_2000, chenciner-venturelli_2000, ferrario-terracini_2004, barrabes-etal_2006, chen_2008, fusco-gronchi-negrini_2011, chen_2012, fenucci-gronchi_2018, fenucci-gronchi_2023}. At the same time, such solutions were numerically computed by implementing discrete minimization techniques \cite{simo_2000, simo_2001,simo_2002,  fenucci-gronchi_2018, ciccarelli-etal_2026}, while usually using a discretized Hessian to verify local minimality properties \cite{berti-etal_2026}.  

A Roundabout-type theory for quadratic functionals with general linear coupled boundary conditions has been developed in \cite{zeidan-zezza_1989, zeidan_1993, zeidan_2001, dosla-dosly_1997, hilscher-zeidan_2008}, expressed either in terms of coupled-points or in the existence of a solution of the Riccati equation with boundary conditions. 
A specialization of the theory to periodic problems, with applications to Celestial Mechanics problems, was presented in \cite{fenucci-jorba_2020, fenucci_2022}, although the characterization of necessary and sufficient conditions was not complete. 
Here we give a full characterization of necessary and sufficient conditions for local minima with periodic boundary conditions, presenting both the approaches based on coupled-points and the existence of particular solutions of the Riccati equation, completing and strengthening the results of \cite{fenucci_2022}.

\section{Definitions of local minimizers}
Fix $T>0$ and define the action functional
\begin{equation}
    \A(u) = \int_0^T L(t,u,\dot{u}) \dt,
\end{equation}
where $L: [0,T] \times \Omega \to \R$ is a $C^3$ function which is $T$-periodic in $t$, $\Omega \subseteq \R^n \times \R^n$ is an open set, and $u:[0,T]\to \R^n$ is a $C^1$ curve. 
Let us define the set of $T$-period curves as
\begin{equation}
    V = \{ u \in C^1([0,T], \R^n) : u(0) = u(T) \}.    
\end{equation}
Given a subset $X \subseteq V$, then a minimum problem with periodic boundary condition consists in determining the minima of $\A$ over the set $X$. Following the classical definitions of minimum points for the fixed end-point problem, we say that $u_0 \in X$ is a:
\begin{itemize}
    \item[(GM)] global minimum point if $\A(u) \geq \A(u_0)$ for all $u \in X$;
    \item[(SLM)] strong local minimum point if there exists $\varepsilon > 0$ such that, for all $u \in X$ with $\lvert\lvert u-u_0 \rvert\rvert_{\infty} < \varepsilon$, then $\A(u) \geq \A(u_0)$;
    \item[(WLM)] weak local minimum point if there exists $\varepsilon > 0$ such that, for all $u \in X$ with $\lvert\lvert u-u_0 \rvert\rvert_{\infty} + \lvert\lvert \dot{u}-\dot{u}_0 \rvert\rvert_{\infty} < \varepsilon$, then $\A(u) \geq \A(u_0)$;
    \item[(DLM)] directional local minimum point if the function $\varphi(s) = \A(u_0 + sv)$ has a local minimum point at $s=0$ for all $v \in V$.
\end{itemize}
From the classical theory of Calculus of Variations, it is well known that if $u_0$ is a $C^2$ minimizer, then it solves the Euler-Lagrange equation associated to $\A$, which is
\begin{equation}
    \frac{\d}{\dt} L_{\dot{u}}(t, u_0(t), \dot{u}_0(t)) =  L_{u}(t, u_0(t), \dot{u}_0(t)).
    \label{eq:EL}
\end{equation}
Since $u_0 \in X$, it also holds that
\begin{equation}
    L_{\dot{u}}(0, u_0(0), \dot{u}_0(0)) = L_{\dot{u}}(T, u_0(T), \dot{u}_0(T)),
    \label{eq:Ludot}
\end{equation}
which leads to $\dot{u}(0) = \dot{u} (T)$, provided that $L$ is globally convex in $\dot{u}$.

Note that $u_0$ is a (DLM) if and only if the second variation of $\delta^2\A$ is non-negative at $u_0$, where $\delta^2\A$ is defined as 
\begin{equation}
    \delta^2 \A(v) = \frac{\d}{\d s}\A(u_0 + s v)|_{s=0} = \int_0^T\big( v(t)\cdot \hat{L}_{uu}(t)v(t) +2\dot{v}(t) \cdot \hat{L}_{u\dot{u}}(t) v(t) + \dot{v}(t)\cdot \hat{L}_{\dot{u} \dot{u}}(t)\dot{v}(t)\big),
\end{equation}
where
\begin{equation}
\begin{split}  
    \hat{L}_{uu}(t) = L_{uu}(t, u_0(t), \dot{u}_0(t)), \\
    \hat{L}_{u\dot{u}}(t) = L_{u\dot{u}}(t, u_0(t), \dot{u}_0(t)), \\
    \hat{L}_{\dot{u} \dot{u}}(t) = L_{\dot{u} \dot{u}}(t, u_0(t), \dot{u}_0(t)).
\end{split}
\end{equation}
The second variation $\delta^2\A$ is a quadratic functional defined on the whole space $V$ of $T$-periodic curves, therefore we are first interested into giving necessary and sufficient conditions for the positivity of these particular functionals.

\section{Positivity of quadratic functionals}

\subsection{Definitions}
Let us consider a quadratic functional 
\begin{equation}
    \Q(v) = \int_0^T \big( v \cdot Pv + 2\dot{v}\cdot Qv + \dot{v}\cdot R\dot{v}\big)\dt,
    \label{eq:quadfunc}
\end{equation}
defined on the space $V$ of $T$-periodic functions, where $P,Q,R:[0,T] \to \R^{n\times n}$ are $C^1$ matrix functions such that $P(t)$ and $R(t)$ are symmetric for all $t\in[0,T]$. The Euler-Lagrange equation associated to $\Q$ is usually called Jacobi differential equation, and it reads as
\begin{equation}
    \frac{\d}{\dt} \big(R\dot{y} +Qy\big) = Q^T \dot{y}+ Py.
    \label{eq:jde_2ndorder}
\end{equation}
Provided that $R(t)$ is non-singular for all $t \in [0,T]$, then we can write Eq.~\eqref{eq:jde_2ndorder} in linear form as
\begin{equation}
    \begin{cases}
        \dot{y} = Ay + Bz, \\
        \dot{z} = Cy - A^Tz,
    \end{cases}
    \label{eq:jde}
\end{equation}
where
\begin{equation}
    A = -R^{-1}Q, \quad B = R^{-1}, \quad C= P - Q^TR^{-1}Q.
\end{equation}
The matrices $B$ and $C$ are also symmetric for all $t\in[0,T]$. We also introduce the matrix form of Eq.~\eqref{eq:jde}:
\begin{equation}
    \begin{cases}
        \dot{Y} = AY + BZ, \\
        \dot{Z} = CY - A^TZ,
    \end{cases}
    \label{eq:jde_matrix}
\end{equation}
where $Y,Z : [0,T] \to R^{n\times n}$ are matrix functions.

\begin{remark}
    Note that if $(Y_1,Z_1)$ and $(Y_2,Z_2)$ are two solutions of Eq.~\eqref{eq:jde_matrix}, then
    \begin{equation}
        Y_1^T(t) Z_2(t) - Z_1^T(t)Y_2(t)\equiv K, 
        \label{rmk:constComb}
    \end{equation}
    where $K\in \R^{n\times n}$ is a constant matrix.
\end{remark}

\begin{definition}
    A solution $(Y,Z)$ of Eq.~\eqref{eq:jde_matrix} is said to be self-conjoined if
    \begin{equation}
        Y^T(t) Z(t) - Z^T(t)Y(t)\equiv 0.
    \end{equation}
\end{definition}

\begin{proposition}
   Let $(Y_0, Z_0)$ be a self-conjoined solution of Eq.~\eqref{eq:jde_matrix} such that $\det
   Y_0(t) \neq 0$ for all $t \in [0,T]$. Then all solutions $(Y,Z)$ of Eq.~\eqref{eq:jde_matrix} are
   given by the formula
   \begin{equation}
      \begin{aligned}
    Y(t) & = Y_0(t)\big( M+S_0(t)N\big) , \\
    Z(t) & = Z_0(t)\big( M+S_0(t)N\big) + Y_0^{-T}(t)N,
 \end{aligned}
   \label{eq:jde2GenSol}
   \end{equation}
   where $M,N \in \R^{n\times n}$ are arbitrary constant matrices and
   \begin{equation}
      S_0(t) = \int_{a}^{t}Y_0^{-1}(s)B(s)Y_0^{-T}(s)\,\d s,
      \label{eq:S0expression}
   \end{equation}
 for some $a \in
   [0,T]$. 
   Moreover
   \begin{equation}
   \begin{aligned}
      Y^T_0(t)Z(t)-Z_0^T(t)Y(t)& = N,  \\
      Y^T(t)Z(t) - Z^T(t)Y(t)  & = M^TN-N^TM.
   \end{aligned}
      \label{eq:MNexpression}
   \end{equation}
   \label{prop:generalSolJDE}
\end{proposition}
\begin{proof}
   Differentiating \eqref{eq:jde2GenSol} and using the fact that $(Y_0, Z_0)$ is a self-conjoined
   solution, we can verify that $(Y,Z)$ defined by \eqref{eq:jde2GenSol} solves
   Eq.~\eqref{eq:jde_matrix}. 
   Indeed, 
   \[
      \begin{split}
         \dot{Y} & = (AY_0 + BZ_0)(M+S_0N)+BY_0^{-T}N \\
         AY + BZ & = AY_0(M+S_0N)+BZ_0(M+S_0N)+BY_0^{-T}N \\
         \dot{Z} & = (CY_0-A^TZ_0)(M+S_0N) +
         Z_0Y_0^{-1}BY_0^{-T}N-Y_0^{-T}(AY_0+BZ_0)Y_0^{-T}N \\
         CY-A^TZ & = CY_0(M+S_0N) - A^T Z_0(M+S_0N)-A^TY_0^{-T}N,
      \end{split}
   \]
   where we have used that 
   \[
      \frac{d}{dt}(Y^{-1})= -Y^{-1}\dot{Y}Y^{-1}. 
   \]
   Moreover, since 
   \[
      Y(a) = Y_0(a)M, \quad Z(a) = Z_0(a)M + Y_0^{-T}(a)N,
   \]
   the matrices $M,N$ can be chosen in one and only one way so that they satisfy arbitrary
   initial conditions. Then \eqref{eq:MNexpression} are easily verified evaluating at the left
   hand sides for $t=a$.
\end{proof}

\begin{definition}
    Let $(y,z)$ be a solution of Eq.~\eqref{eq:jde} with $y(0)=0$. A point $c \in (0,T]$ is said to be conjugate with zero if $y(c)=0$.
\end{definition}

\begin{definition}
    Let $(y,z)$ be a solution of Eq.~\eqref{eq:jde}. A point $c \in (0,T]$ is said to be semi-coupled with $0$ if
    \begin{equation}
        \begin{cases}
            y(0) = y(c), \\
    \displaystyle        z(0) = z(c) + \int_c^T P(t)\dt \ y(c).
        \end{cases}
    \end{equation}
    Moreover, if $y(t) \neq y(c)$ for all $t \in (t,T]$, $c$ is said to be coupled with $0$.
\end{definition}

\noindent We introduce the following useful for later discussions.
\begin{itemize}
    \item[(L)] The Legendre condition holds if $R(t)\geq 0$ for all $t\in[0,T]$.
    \item[(L')] The strengthened Legendre condition holds if $R(t) > 0$ for all $t\in[0,T]$.
    \item[(R)] The regularity condition holds if $v \cdot \int_0^T P(t)\dt \ v \geq 0$ for all $v \in \R^n$.
    \item[(R')] The strengthened regularity condition holds if $v \cdot \int_0^T P(t)\dt v > 0$ for all $v \in \R^n \setminus\{0\}$.
    \item[(J)] The Jacobi condition (J) holds if every non-zero solution $(y,z)$ of Eq.~\eqref{eq:jde} with $y(0)=0$ does not have any conjugate point in $(0,T)$.
    \item[(J')] The strengthened Jacobi condition (J) holds if every non-zero solution $(y,z)$ of Eq.~\eqref{eq:jde} with $y(0)=0$ does not have any conjugate point in $(0,T]$.
    \item[(C)] Condition (C) holds if every non-zero solution $(y,z)$ of Eq.~\eqref{eq:jde} does not have any semi-coupled point in $(0,T)$.  
    \item[(C')] Condition (C') holds if every non-zero solution $(y,z)$ of Eq.~\eqref{eq:jde} does not have any semi-coupled point in $(0,T]$. 
\end{itemize}

We associate to the quadratic functional of Eq.~\eqref{eq:quadfunc} the Riccati matrix equation, which is defined as
\begin{equation}
   \dot{W}-C+WA+A^TW+WBW = 0.
   \label{eq:rde}
\end{equation}
For this purpose we note, that if $(Y,Z)$ is a solution of Eq.~\eqref{eq:jde_matrix} such that $Y(t)$ is non-singular on $[0,T]$, then $W(t) =
Z(t)Y^{-1}(t)$ is a solution of Eq.~\eqref{eq:rde} defined on the whole
interval $[0,T]$. In fact
\begin{align*}
   \frac{\d Z}{\dt} Y^{-1} + Z \frac{\d Y^{-1}}{\dt}
   -C +ZY^{-1}A+A^TZY^{-1}+ZY^{-1}BZY^{-1}     & =\\
    (CY-A^TZ)Y^{-1}-ZY^{-1}\dot{Y}Y^{-1}-C +ZY^{-1}A+A^TZY^{-1}+ZY^{-1}BZY^{-1}   &=  \\
   -ZY^{-1}\dot{Y}Y^{-1} +ZY^{-1}(A+BZY^{-1})  &= \\
   -ZY^{-1}\dot{Y}Y^{-1} +ZY^{-1}\dot{Y}Y^{-1} &= 0  \\
\end{align*}
Moreover, if $(Y,Z)$ is also self-conjoined, then $W$ is symmetric.

\subsection{Approach with coupled points}

First we focus on necessary conditions to have a non-negative quadratic functional. In the classical setting of the fixed endpoints problem, it is known that (L) is a necessary condition and moreover, if (L') holds, then (J) is also necessary. The interested reader can refer to classical books of Calculus of Variations, such as \cite{giaquinta_2004}, for detailed proofs of these facts. 
Here we can prove similar statements, as done in the following lemmas. 

\begin{lemma}
   If $\Q(v)\geq 0$ for all $v \in V$, then conditions (L) and (R) hold. 
   \label{lemma:necessary1}
\end{lemma}
\begin{proof}
   The proof that (L) holds is the same as in the fixed endpoints problem, since we can restrict ourself to local variations vanishing at the extrema of the interval. The regularity condition (R) follows by taking constant variations $v(t) \equiv v_0 \in \R^n $, because
   \[
      \Q(v) = v_0 \cdot \int_{0}^{T}P(t)\, \dt \, v_0 \geq 0,
   \]
   and this is exactly condition (R).
\end{proof}

\begin{lemma}
   If $\Q(v) \geq 0$ for all $v \in V$ and condition (L') holds, then (C) holds.
   \label{lemma:necessary2}
\end{lemma}
\begin{proof}
   Suppose that the system of Eq.~\eqref{eq:jde} has a semi-coupled point, i.e. there exist a non-zero solution $(y,z)$ such that for a certain
   $c \in (0,T]$ we have $y(c) = y(0)$ and 
   \[
      z(0) = z(c) + \int_{c}^{T} P(t) \dt \, y(c).
   \]
   Let us define 
   \[
      v_c(t) =
      \begin{cases}
         y(t) & t \in [0,c], \\
         y(c) & t \in [c,T].
      \end{cases}
   \]
   Note that $v_c$ is $T$-periodic, then evaluating the quadratic functional and integrating by
   parts we obtain
   \[
      \begin{split}
         \Q(v_c) & = \int_{0}^{c}\big( y \cdot  P y + 2 \dot{y} \cdot Q y + \dot{y} \cdot
         R \dot{y} \big) \dt + y(c) \cdot \int_{c}^{T} P(t) \, \dt \, y(c) \\
         & = \int_{0}^{c}\big( y \cdot (P y + Q^T \dot{y}) + \dot{y} \cdot (
         Q y + R \dot{y}) \big) \dt + y(c) \cdot \int_{c}^{T} P(t) \, \dt \, y(c)  \\
         & = \int_{0}^{c}\big( y \cdot \dot{z} + \dot{y} \cdot z \big) \dt + y(c) \cdot
         \int_{c}^{T} P(t) \, \dt \, y(c) \\
         & = \big[ y\cdot z \big]_0^c + y(c) \cdot \int_{c}^{T} P(t) \, \dt \, y(c) \\
         & = y(c)\cdot\bigg( z(c)- z(0) +  \int_{c}^{T} P(t) \, dt \, y(c) \bigg) \\
         & = 0.
      \end{split} 
   \]
   This means that $v_c$ is a minimizer of $\Q$. However, by condition (L'), the
   minimizers solve Eq.~\eqref{eq:jde} and, moreover, the
   theorem of uniqueness of the solutions holds. In this case, we would obtain that $v_c(t)$ is constant, which is a contradiction.
\end{proof}

Let us prove the following two lemmas which are useful for later use.
\begin{lemma}
    Let $c \in (0, T]$ and define $N(c) = \int_0^T P(t)\ \dt$. Let us define also
    \[
    \begin{split}
        V^c & =  \{ w \in C^1([0,c], \R^n) : w(0)=w(c)\}, \\
        V_c & = \{ v \in V : v \equiv v(c) \textrm{ on } [c,T] \}.
    \end{split}
    \]
    Define also
    \[
    \Q_c(w) = \int_0^c \big( w\cdot P w + 2 \dot{w} \cdot Q w + \dot{w
    }\cdot R \dot{w} \big)\dt + w(c) \cdot N(c) w(c).
    \]
    Then:
    \begin{itemize}
        \item[(a)] The function $v \mapsto v|_{[0,c]}$ is a bijection between $V_c$ and $V^c$, and $\Q(v) = \Q_c(v|_{[0,c]})$.
        \item[(b)] $\ker \Q_c \neq \{ 0 \}$ if and only if $c$ is semi-coupled with 0.
    \end{itemize}
    \label{lemma:lemmaA}
\end{lemma}
\begin{proof}
    The proof follows directly from the proof of Lemma~\ref{lemma:necessary2}.    
\end{proof}

\begin{lemma}
    Let (L') and (R') hold. Then there exists $c_0$ and $\varepsilon>0$ such that $\Q_c(w) \geq \varepsilon \big(|w(0)|^2 + \norm{\dot{w}^2}_{L^2(0,c)}\big)$ for all $c \in (0, c_0]$ and $w \in V^c$.
    \label{lemma:lemmaB}
\end{lemma}
\begin{proof}
    Let $R \geq k\Id$ and $\mu = \lambda_{\min}(\int_0^TP \ \dt) >0$ because (R') holds. Let $w = w_0 + h$, where $w_0 = w(0) = w(c)$ and $h(0)=h(c)=0$. Then
    \[
    \int_0^c w_0 \cdot P w_0 \ \dt + w_0 \cdot N(c) w_0 = w_0 \cdot \int_0^TP(t) \dt \ w_0 \geq \mu |w_0|^2.
    \]
    Moreover, $\dot{h}=\dot{w}$ and because of the Poincar\'e inequality, we have $\norm{h}_{L^2} < (c/\pi)\norm{\dot{h}}_{L^2}$. Therefore
    \[
    \Q_c(w)\geq \mu |w_0|^2 + \bigg[ k - \norm{P}_\infty \frac{c^2}{\pi^2} - 2 \norm{Q}_\infty \frac{c}{\pi} \bigg] \norm{\dot{h}}^2 - 2 \bigg[ \norm{Q}_\infty\sqrt{c} + \norm{P}_\infty \frac{c^{3/2}}{\pi} \bigg] |w_0| \norm{\dot{h}}.
    \]
    For small $c$, the first term in the brackets is larger than $k/2$, while the second term is larger than $K\sqrt{c}$. From the inequality $2K \sqrt{c} |w_0| \norm{\dot{h}} \leq (\mu/2) |w_0|^2 + (2K^2c/\mu)\norm{\dot{h}}^2$ it follows that
    \[
    \Q_c(w) \geq \frac{\mu}{2}|w_0|^2 + \bigg( \frac{k}{2} - \frac{2K^2c}{\mu} \bigg)\norm{\dot{h}}^2,
    \]
    and the term in the round brackets is positive for small $c$.
\end{proof}

The following is an analogue of the Reid Roundabout Theorem for the fixed-end problem, but specialized to the case of periodic boundary conditions. A more general result obtained for general boundary conditions is given in \cite{dosla-dosly_1997}, but an independent proof specialized for the periodic boundary conditions is given here. 
\begin{theorem}
   Assume that condition $\rm{(L')}$ holds. Then the following properties are equivalent
   \begin{itemize}
      \item[(i)] $\Q(v)> 0 $ for all non-zero $v \in V$.
      \item[(ii)] Conditions $\rm{(C')}$ and $\rm{(R')}$ hold.
      \item[(iii)] Condition $\rm{(J')}$ holds and 
         \begin{equation}
    W_0(T)-W_T(0)-Y_0^{-1}(T)-Y_0^{-T}(T) > 0,
            \label{eq:condIIIperiodic}
         \end{equation}
         where $(Y_0, Z_0)$ and $(Y_T, Z_T)$ are the solutions of Eq.~\eqref{eq:jde_matrix} with initial conditions
         \begin{equation}
             \begin{cases}
                 Y_0(0) = 0, \\
                 Z_0(0) = \Id, 
             \end{cases}
             \quad             
             \begin{cases}
                 Y_T(T) = 0, \\
                 Z_T(T) = -\Id, 
             \end{cases}
         \end{equation}
         and $W_0 = Z_0Y_0^{-1}, \ W_T = Z_TY_T^{-1}$.
   \end{itemize}
   \label{th:DoslaSufCond}
\end{theorem}
\begin{proof}
(i)$\Rightarrow$(ii) This follows using similar proofs of Lemma~\ref{lemma:necessary1} and Lemma~\ref{lemma:necessary2}.

(i)$\Rightarrow$(iii) Since $\Q >0$ on any non zero $v \in V$, in particular $Q>0$ for any $v \in V_0 = \{ v \in V: v(0)=v(T)=0\}$, which is the variation space for the classical fixed end-point problem. Therefore (J') holds, and hence $Y_0, Y_T$ are invertible in $(0,T]$ and $[0,T)$, respectively. Let $x\in \R^n$ and define
\[
v^*(t) = \big( Y_0(t)Y_0(T)^{-1} + Y_T(t)Y_0(T)^{-T} \big)x.
\]
Then $v^*(0)=v^*(T) = x$ and, since $(Y_0, Z_0), (Y_T, Z_T)$ solve Eq.~\eqref{eq:jde_matrix}, then 
\[
\Q(v^*) = x \cdot \big( W_0(T)-W_T(0)-Y_0^{-1}(T)-Y_0^{-T}(T) \big)x,
\]
which is positive definite by hypothesis. 

(iii)$\Rightarrow$(i) Let $v\in V$ and set $v_0 = v(0)=v(T)$. Let us define 
\[
v^*(t) = \big( Y_0(t)Y_0(T)^{-1} + Y_T(t)Y_0(T)^{-T} \big)v_0.
\]
Set $h = v-v^*$, and note that $h \in V_0$, i.e. it vanishes at the extrema. Then we can write, by bi-linearity:
\[
\Q(v) = \Q(v^*) + 2 \mathcal{B}(v^*,h) + \Q(h).
\]
Since $Y_0, Y_T$ are solutions of Eq.~\eqref{eq:jde_matrix} and $h\in V_0$, by integrating by parts we obtain $\mathcal{B}(v^*,h)=0$. Moreover, using the same argument of (i)$\Rightarrow$(iii), we have
\[
\Q(v^*) = v_0 \cdot \big( W_0(T)-W_T(0)-Y_0^{-1}(T)-Y_0^{-T}(T) \big)v_0.
\]
This is therefore positive by assumption. In addition, $\Q(h)>0$ because $h$ is a competitor for the fixed end-point problem and (J') holds, therefore $\Q(v)>0$ follows.

(ii)$\Rightarrow$(i)
Let us define 
\[
H = \{ \hat{w} \in H^1([0,1], \R^n): \hat{w}(0)=\hat{w}(1) \},
\]
and for $\hat{w} \in H$ set $w(t) = \hat{w}(t/c)$. We define also
\[
F_c(\hat{w}) = \Q_c(w)=\int_0^1\big[ c \hat{w}\cdot P(c\tau)\hat{w} + 2\hat{w}' \cdot Q(c\tau)\hat{w} + \frac{1}{c}\hat{w}'\cdot R(c\tau)\hat{w}' \big]\d\tau +\hat{w}(1)\cdot N(c) \hat{w}(1).
\]
For $c \in [c_0, T]$, where $c_0$ is as in Lemma~\ref{lemma:lemmaB}, the coefficients depends continuously in $c$ and uniformly in $\tau$, therefore $c \mapsto F_c$ is continuous with respect to all the bounded quadratic forms on $H$. Let
\[
\lambda(c)=\inf \{ F_c(\hat{w}) : \norm{ \hat{w}}_H = 1 \}.
\]
Then $|\lambda(c)-\lambda(c')| \leq \norm{ F_c - F_{c'}}$, hence $\lambda$ is a continuous function. In addition, from condition (L'), we have
\[
F_c(\hat{w}) \geq \frac{k}{2T} \norm{\hat{w}'}^2_{L^2} - C \norm{\hat{w}}_{L^2}.
\]
Let us now prove that if $\lambda(c)=0$ then $\ker F_c \neq \{ 0 \}$. Let $\{ \hat{w}_j\} \subseteq H$ be a sequence such that $\norm{\hat{w}_j}_H = 1$ and $F_c(\hat{w}_j) \to 0$. Up to subsequences, there exists $\hat{w}$ such that $\hat{w}_j \to \hat{w}$ in $H$ and $\hat{w}_j \to \hat{w}$ in $L^2$ and in $C^0$. The term $\int \hat{w}' \cdot R \hat{w}'$ is convex and therefore weakly lower semi-continuous, while the other terms are weakly continuous. Therefore $F_c(\hat{w}) \leq \liminf_j F_c(\hat{w}_j) = 0$. On the other hand, $\lambda(c)=0$ implies $F_c \geq 0$, therefore $F_c(\hat{w})=0$. Moreover $\hat{w} \neq 0$ because $\norm{\hat{w}_j}_H = 1$, therefore $\ker F_c \neq \{ 0 \}$. 

Let us now suppose that there exists a non-zero $v\in V$ such that $\Q(v) \leq 0$, hence $\lambda(T) \leq 0$. Because of Lemma~\ref{lemma:lemmaB}, $\lambda(c)>0$ for $c \in (0, c_0]$. Let
\[
c^* = \inf \{c \in (0,T] : \lambda(c)\leq 0 \} \in [c_0, T].
\]
By continuity, $\lambda(c^*)=0$, then $\ker F_{c^*} \neq \{0\}$ and by Lemma~\ref{lemma:lemmaA} $c^* \in (0,T]$ is a semi-coupled point with 0, which is in contradiction with the hypothesis. Therefore $\Q(v) > 0$ for all non-zero $v \in V$.
\end{proof}

\begin{remark}
From a numerical point of view, some conditions of Theorem~\ref{th:DoslaSufCond} are easier to check than others. It is clear that (R') can be checked by numerically computing an integral, and check if the resulting matrix is positive definite by computing its eigenvalues. However, condition (C') cannot be fully checked numerically, since it takes into account all the solutions of Eq.~\eqref{eq:jde}. 
On the other hand, conditions (J') and Eq.~\eqref{eq:condIIIperiodic} can be managed numerically. In fact, one can
\begin{itemize}
      \item[(1)] compute the solution $(Y_0,Z_0)$ of Eq.~\eqref{eq:jde_matrix} by numerical integration, using the initial conditions $Y_0(0)=0, \,
         Z_0(0) = \Id$ on $[0,T]$. Then
         check whether $\det Y_0(t) \neq 0$ for all $t \in (0,T]$ to check if (J') holds;
      \item[(2)] compute the solution $(Y_T, Z_T)$ of Eq.~\eqref{eq:jde_matrix} by numerical integration on $[0,T]$ backwards in time, using the initial  conditions $Y_T(T) = 0, \, Z_T(T)=-\Id$;
      \item[(3)] invert the matrices $Y_0(T), \, Y_T(0)$, compute $W_0(T), W_T(0)$ and construct the matrix in Eq.~\eqref{eq:condIIIperiodic}; 
      \item[(4)] compute the eigenvalues of the matrix in Eq.~\eqref{eq:condIIIperiodic} and check whether they are all positive or not.
\end{itemize}
\end{remark}

\subsection{Approach with Riccati equation}

We describe here an approach based on the existence of solutions of the Riccati equation with particular boundary conditions, similar to that presented in \cite{zeidan_1993, zeidan_2001}. 
We introduce the following condition. 
\begin{itemize}
   \item[(SR)] Condition (SR) is satisfied if there exists a symmetric
      solution $W(t)$ of the Riccati equation \eqref{eq:rde}, defined on the whole interval
      $[0,T]$, such that
      \begin{equation}
         W(T)-W(0) > 0.
         \label{eq:riccatiBC}
      \end{equation}
\end{itemize}
This condition is sufficient to have a positive definite quadratic functional in
the case of periodic boundary conditions. 
\begin{theorem}
   Let conditions $\rm{(L')}$ and $\rm{(SR)}$ hold. Then $\Q(v) > 0$ for all
   non-zero $v \in V$. 
   \label{th:DLMsuf}
\end{theorem}
\begin{proof}
   The proof was presented already in \cite{fenucci_2022}, but we report it here for completeness.
   Given $v \in V, \ v\neq0$, then
   \[
      \begin{split}
         \Q(v) & = \int_{0}^{T}\bigg( v \cdot P v + 2 \dot{v} \cdot Q v + \dot{v} \cdot R
         \dot{v} - \frac{\d}{\dt}(v \cdot W v) + \frac{\d}{\dt}(v \cdot W v) \bigg) \dt \\
         & = v(0)\cdot\big(W(T)-W(0)\big)v(0) + \int_{0}^{T}\big(
         R\dot{v}+Qv-Wv\big)\cdot B\big( R\dot{v}+Qv-Wv\big)\dt ,
      \end{split}
   \]
   where we used $v(0)=v(T)$. Since (L') holds, $B$ is positive definite for all $t \in [0,T]$, then the function in the integral
   is non-negative. Since $W(T)-W(0)$ is positive definite by hypothesis, we have  $\Q(v) > 0$. 
\end{proof}

\begin{proposition}
    Let (L') and (J') hold. Let $Y_1= Y_T Y_0(T)^{-T}$. Then $W$ is a symmetric solution of the Riccati equation defined on the whole interval $[0,T]$ if and only if it is of the form $W=ZY^{-1}$, with
    \[
    \begin{cases}
            Y = Y_0 + Y_1S, \\
            Z = Z_0 + Z_1S,
    \end{cases}
    \]
    and $S=W(0)-W_T(0)^{-1}$ is a positive definite symmetric matrix. Moreover, we have 
    \begin{equation}
        W(T)-W(0)= W_0(T)-W_T(0) - Y_0(T)^{-T}SY_0(T)^{-1} - S^{-1}.
        \label{eq:generalRiccatiSolution}
    \end{equation}
    \label{prop:RiccatiCharacterization}
\end{proposition}
\begin{proof}
    Let us assume (SR) and be $W$ the corresponding solution of the Riccati equation. Every solution of Eq.~\eqref{eq:jde_matrix} can be written as $(Y_0A + Y_TN, Z_0A + Y_TN)$. By setting $Y(0)=0$ we get $N = Y_T(0)^{-1}$ and $A=W(0)-W_T(0)$. Let us consider then the solution $(Y,Z) = \big((Y_0A + Y_TN)S, (Z_0A+Z_TN)S \big)$, where $S = (W(0) - W_T(0))^{-1}$. It is easy to verify that $(Y,Z)$ is self-conjoined. Let us assume that $W' = ZY^{-1}$ is defined on $[0,T]$. Then 
    \begin{equation}
            \begin{split}
        W'(0) & = S^{-1} + W_T(0), \\
        W'(T) & = W_0(T) - Y_0(T)^{-T}SY_0(T)^{-1}.
    \end{split}
    \label{eq:newRiccati}
    \end{equation}
    Let us now prove that $Y$ is invertible. By Proposition~\ref{prop:generalSolJDE}, $Y_0(t) = Y_1(t)S_1(t)$, where $(Y_1, Z_1) = (Y_T Y_0(T)^{-T}, Z_TY_0(T)^{-T})$. Then $Y(t)=Y_1(t)(S + S_1(t))$.
    Since $B>0$ and $Y_1(T)=0$, we have $Y_1^{-1}BY_1^{-T} \sim K/(t-T)^2$, with $K>0$, hence $S_1(t) \to + \infty$ in every direction for $t\to T^{-}$. If $S$ had a negative eigenvalue, then the number of negative eigenvalues of $S_1(t) + S$ would pass from $\geq 1$ to 0, and by continuity also some eigenvalue would vanish in $(0,T)$, hence $Y$ would be singular. This is however not possible because (J') holds, hence $S>0$ and, viceversa, if $S>0$ then $S_1+S > 0$.

    Conversely, if $S$ is a positive definite symmetric matrix, then $W'$ as defined before is a symmetric solution of the Riccati equation defined on the whole $[0,T]$ satisfying Eq.~\eqref{eq:riccatiBC} and such that Eq.~\eqref{eq:generalRiccatiSolution} holds.
\end{proof}
\begin{corollary}
    Let (L') and (J') hold. Then (SR) holds if and only if there exists a positive definite symmetric matrix $S$ such that $W_0(T) - W_T(0) > Y_0(T)^{-T}SY_0(T)^{-1} + S^{-1}$.
    \label{cor:SRcharacterization}
\end{corollary}

\begin{remark}
Note that one could try to match the sufficient conditions of Theorem \ref{th:DoslaSufCond} point \emph{(iii)} with the (SR) condition on the existence of a symmetric solution of the Riccati equation with boundary conditions. In fact, these conditions are equivalent for the classical case of the fixed end-point problem, since it is known that the absence of conjugate points leads to the existence of a symmetric solution of the Riccati equation. More generally, the same can be done when the endpoints are disjoined, see \cite{zeidan_1993}. 
   Note that the matrices in Eq.~\eqref{eq:condIIIperiodic} and in Eq.~\eqref{eq:generalRiccatiSolution} look
   similar. Therefore, one could try to deduce the (SR) condition assuming that (J') holds and
   Eq.~\eqref{eq:condIIIperiodic} holds, by matching the matrix of Eq.~\eqref{eq:condIIIperiodic} with that of Eq.~\eqref{eq:generalRiccatiSolution}.
   These are the same if an only if $S=Y_0(T)$, and condition (SR) holds if and only if $Y_0(T)$ is symmetric and positive definite. However, this last condition is generally not true.
   \label{rmk:rmk1}
\end{remark}

\begin{remark}
        The only exception to Remarks~\ref{rmk:rmk1} is for dimension $n=1$. In fact, the transpose and the order of the multiplications are not relevant in dimension one, hence the (SR) condition is equivalent to (J') together with Eq.~\eqref{eq:condIIIperiodic}, and to all the other equivalent conditions of Theorem~\ref{th:DoslaSufCond}. 
   \label{rmk:riccatiProblems}
\end{remark}

\subsubsection{(SR) is not a necessary condition}
We now give a counterexample to show that the positivity of a quadratic functional $\Q$ does not imply (SR).
Let us set $T=1$, $n=2$, and
\[
J = \begin{bmatrix}
    0 & -1\\
    1 & 0
\end{bmatrix}.
\]
Let us also set $\omega = 2, k = \pi/2$ and take
\[
L(u,\dot{u}) = \frac{1}{2}\norm{\dot{u}+\omega J u}^2 - \frac{1}{2}k^2\norm{u}^2.
\]
In physical terms, this is the Lagrangian describing the motion of a particle in a rotating system moving under a repulsing potential. 
Note that $u_0 \equiv 0$ is a periodic solution of the Euler-Lagrange equation associated to the functional $\A(u)=\int_0^1L(t,u,\dot{u})\dt$. The second variation along $u_0$ is:
\[
\Q(v) = \int_0^1 \big( \norm{\dot{v}}^2 +2 \omega\ \dot{v}\cdot Jv + (\omega^2-k^2)\norm{v}^2 \big)\dt,
\]
hence $R=\Id, Q=\omega J, P=(\omega^2-k^2)\Id$. Therefore conditions (L') and (R') hold. The Jacobi equation is therefore
\[
\begin{cases}
    \dot{Y} = -\omega J Y + Z, \\
    \dot{Z} = -k^2Y + \omega J^TZ.
\end{cases}
\]
By using the substitution $(Y', Z') = (e^{-\omega J t}Y, e^{-\omega J t}Z)$, the Jacobi equation becomes 
\[
\begin{cases}
    \dot{Y}' = Z', \\
    \dot{Z}' = -k^2Y'.
\end{cases}
\]
Therefore the general solution of the Jacobi equation in these variables is
\[
\begin{cases}
    Y'(t) = A \cos(kt) + B \sin(kt), \\
    Z'(t) = -Ak \sin(kt) + Bk\cos(kt), 
\end{cases}
\]
hence it can be computed that 
\[
\begin{cases}
    \displaystyle Y_0(t)=\frac{\sin(kt)}{k}e^{-\omega Jt}, \\
    Z_0(t)=\cos(kt) e^{-\omega Jt}.
\end{cases}
\]
Then $\det Y_0(t) = \sin^2(kt)/k$, which is positive on the whole $(0,1]$ because $k=\pi/2 < \pi$. Therefore the strengthened Jacobi condition (J') holds.
In addition, 
\[
\begin{cases}
    \displaystyle Y_T(t) = \frac{\pi}{2}\cos(kt)e^{\omega J(T-t)}, \\
    \displaystyle Z_T(t) = -\frac{\pi}{2}k\sin(kt)e^{\omega J(T-t)}.
\end{cases}
\] 
Therefore $W_0(T) = 0$ and $W_T(0)=0$, while $Y_0(T)^{-1} = \frac{\pi}{2}e^{\omega J}$, hence $-(Y_0(T)^{-1} + Y_0(T)^{-T}) = -\pi \cos(2) \Id$, which is positive definite since $\cos(2) \approx -0.4161$. Therefore Eq.~\eqref{eq:condIIIperiodic} holds, and by Theorem~\ref{th:DoslaSufCond} the quadratic functional is positive definite.

On the other hand, the condition (SR) is not satisfied.  In fact, by Corollary~\ref{cor:SRcharacterization} we would need $Y_0(T)^{-T}SY_0(T)^{-1} + S^{-1} < 0$, but this is not possible because $S>0$ is required. Therefore, there is no symmetric positive definite solution $W$ to the Riccati equation defined on $[0,T]$ such that $W(T)-W(0)>0$.

\section{Weak local minimizers}
\label{s:wlm}
Using the theory of quadratic functionals, we can find sufficient conditions for (WLM) for the case of periodic boundary conditions. 
\begin{lemma}
   Let $\Q$ be a quadratic functional defined on $V$. Then there exists $C>0$ such
   that
   \begin{equation}
      \lvert \Q(v) \rvert \leq C \int_{0}^{T}\bigg( \lvert v \rvert^2 
      + \lvert \dot{v} \rvert^2 \bigg) \dt,
      \label{eq:quadIneq1}
   \end{equation}
   for all $v \in V$.
   \label{lemma:WLM1}
\end{lemma}
\begin{proof}
   It is sufficient to note that
   \[
      2\lvert v \rvert \lvert \dot{v} \rvert \leq \lvert v \rvert^2 + \lvert
      \dot{v} \rvert^2,
   \]
   Then we have that
   \[
      \lvert \Q(v) \rvert \leq C_1 \int_{0}^{T} \lvert \dot{v} \rvert^2 \dt \, + C_2
      \int_{0}^{T} \lvert v \rvert^2 \dt,
   \]
   where the constants $C_1, C_2$ are
   \[
      C_1 = \norm{Q}_\infty+\norm{R}_\infty, \quad C_2 = \norm{Q}_\infty+\norm{P}_\infty.
   \]
   Then taking $C=\max\{C_1,C_2\}$ the thesis follows.
\end{proof}
\begin{remark}
   Note that the coefficient $C$ is small if $C_1$ and $C_2$ are both small. This happens
   when the matrix functions $P,Q,R$ are small in the sup norm. 
   \label{rmk:smallCoeff}
\end{remark}
\begin{lemma}
   Let $\Q$ be a positive definite quadratic functional defined on $V$ and satisfying
   $\rm{(L')}$. Then there exists
   $\varepsilon_0>0$ such that
   \begin{equation}
      \Q(v) \geq \varepsilon_0 \int_{0}^{T}\bigg(\lvert v \rvert^2 + \lvert \dot{v}
      \rvert^2\bigg)\dt,
      \label{eq:quadIneq2}
   \end{equation}
   for all $v \in V$.
   \label{lemma:WLM2}
\end{lemma}
\begin{proof}
   Consider the functional
   \[
      \begin{split}
      \hat{\mathcal{Q}}_\varepsilon(v) & := \Q(v) -  \varepsilon \int_{0}^{T}\bigg(\lvert v
      \rvert^2 + \lvert \dot{v} \rvert^2\bigg)\dt \\
      & = \int_{0}^{T}\bigg( v\cdot(P-\varepsilon \Id) v + 2 \dot{v}\cdot Q v +
         \dot{v}\cdot
      (R-\varepsilon \Id) \dot{v} \bigg) \dt \\
      & = \int_{0}^{T}\bigg( v\cdot \hat{P} v + 2 \dot{v}\cdot \hat{Q} v +
      \dot{v}\cdot \hat{R} \dot{v} \bigg) \dt,
   \end{split}
   \]
   where $\hat{P} = P-\varepsilon\Id, \, \hat{Q}=Q, \, \hat{R}=R-\varepsilon\Id$.
   We verify that $\hat{\mathcal{Q}}_{\varepsilon}$ is still positive definite for $\varepsilon$
   small enough, using the equivalent conditions given by Theorem
   \ref{th:DoslaSufCond}. For $\varepsilon$ small enough, 
   $\hat{\mathcal{Q}}_{\varepsilon}$ verifies (L'), (J') and Eq.~\eqref{eq:condIIIperiodic}:
   \begin{itemize}
      \item[(L')] holds since, from condition (L') on $R$, we have that $R \geq k\Id$ for
         some $k >0$. Hence $\hat{R} \geq (k-\varepsilon)\Id$. For $\varepsilon$ small enough we have
         that $k-\varepsilon>0$. 
      \item[(J')] holds because of the uniform convergence of the solutions of Eq.~\eqref{eq:jde_matrix} as $\varepsilon \to
         0$. 
      \item[\eqref{eq:condIIIperiodic}] holds for the same reason above, since the eigenvalues are continuous 
         functions with respect to the entries of the matrix.
   \end{itemize}
   Therefore, for $\varepsilon$ small enough, the quadratic functional
   $\hat{\mathcal{Q}}_\varepsilon$ is positive definite and the thesis follows.
\end{proof}
\begin{lemma}
   Let $\Q_0$ be a positive definite quadratic functional defined on $V$, with coefficients $P_0, Q_0, R_0$, satisfying
   $\rm{(L')}$. Let $\Q$ be a quadratic functional defined on $V$ with coefficients $P,Q,R$. If
   \[
      \norm{P-P_0}_{\infty}, \quad \norm{Q-Q_0}_{\infty}, \quad \norm{R-R_0}_{\infty},
   \]
   are small enough, then $\Q$ is positive definite.
   \label{lemma:WLM3}
\end{lemma}
\begin{proof}
   Let us write
   \[
      \Q(v) = \Q_0(v) + \hat{\Q}(v),
   \]
   where
   \[
      \hat{\Q}(v) = \int_{0}^{T}\bigg( v\cdot(P-P_0)v +
      2\dot{v}\cdot(Q-Q_0)v + \dot{v} \cdot (R-R_0) \dot{v} \bigg) \dt.
   \]
   From Lemma \ref{lemma:WLM2} we have that
   \[
      \Q_0(v) \geq \varepsilon_0 \int_{0}^{T}\bigg(\lvert v \rvert^2 + \lvert \dot{v}
      \rvert^2\bigg)\dt.
   \]
   From Lemma \ref{lemma:WLM1}, there exists $C$ such that
   \[
      \lvert \hat{\Q}(v) \rvert \leq C \int_{0}^{T}\bigg(\lvert v \rvert^2  +
      \lvert \dot{v} \rvert^2 \bigg) \dt.
   \]
   Then 
   \[
      \Q(v) \geq \varepsilon_0 \int_{0}^{T}\bigg(\lvert v \rvert^2 + \lvert \dot{v}
      \rvert^2\bigg)\dt - C \int_{0}^{T}\bigg( \lvert v \rvert^2 + \lvert \dot{v} \rvert^2\bigg) \dt,
   \]
   and the coefficient $C$, from Remark \ref{rmk:smallCoeff}, is smaller than $\varepsilon_0$ for 
   $\norm{P-P_0}_{\infty}$, $\norm{Q-Q_0}_{\infty}$, $\norm{R-R_0}_{\infty}$ small enough,
   hence $\Q$ is positive definite.
\end{proof}
\begin{theorem}
   Let $u_0 \in X$ be a $C^2$ periodic solution of the Euler-Lagrange equation \eqref{eq:EL}, and with a positive second variation $\delta^2\A[u_0]$. Then $u_0$ is a (WLM).
   \label{th:WLMsuf}
\end{theorem}
\begin{proof}
   Let $u \in X$ and set $v = u-u_0 \in V$. Consider 
   \[
      \varphi(s) := \A(u_0 + sv).
   \]
   Then
   \[
      \A(u) = \varphi(1) =  \varphi(0) + \varphi'(0) + \frac{1}{2}\varphi''(s_0) = \A(u_0)
      +\frac{1}{2}\varphi''(s_0), \quad s_0 \in (0,1).
   \]
   We have that
   \[
      \begin{split}
         \varphi''(s_0)  = \int_{0}^{T}&\bigg( v(t)\cdot L_{uu}\big(t, u_0(t)+s_0 v(t),
         \dot{u}_0(t)
         +s_0\dot{v}(t)\big) v(t)\\
         &+
         2\dot{v}(t)\cdot L_{\dot{u}u}\big(t, u_0(t)+s_0 v(t), \dot{u}_0(t)
         +s_0\dot{v}(t)\big) v(t) \\& +
         \dot{v}(t)\cdot L_{\dot{u}\dot{u}}\big(t, u_0(t)+s_0 v(t), \dot{u}_0(t)
         +s_0\dot{v}(t)\big)
         \dot{v}(t) \bigg) \dt. \\
      \end{split}
   \]
   Note that $\varphi''(s_0)$ is a quadratic functional. If $\norm{v}_\infty+\norm{\dot{v}}_\infty$ is small enough,
   by the continuity of the second derivatives of $L$, the coefficients of $\varphi''(s_0)$ are near to the
   coefficients of the second variation associated to $u_0$, in the sup norm. Therefore, from 
   Lemma \ref{lemma:WLM3} we conclude that $\varphi''(s_0)\geq 0$ and $u_0$ is a (WLM).
\end{proof}

\section{Strong local minimizers}
\label{s:slm}
To discuss the strong local minimizers, we need additional definitions and conditions. We
define the Weierstrass excess function as
\begin{equation}
   E(t,u,v,w) := L(t,u,w) - L(t,u,v) - (w-v) \cdot L_{\dot{u}}(t,u,v).
   \label{eq:wef}
\end{equation}
Let $u_0 \in X$ be a solution of the Euler-Lagrange equation \eqref{eq:EL}. To simplify
the notations in the following, we define the tube around $u_0$ of radius
$\varepsilon>0$ as
\[
   T(u_0, \varepsilon) = \big\{ (t, y) \in [0,T] \times \R^n : |y-u_0(t)| < \varepsilon
   \big\}, 
\]
and the restricted tube as
\[
   RT(u_0, \varepsilon) = \big\{ (t, y, v) \in [0,T] \times \R^n \times \R^n : |y-u_0(t)| <
      \varepsilon, \, |v-\dot{u}_0(t)|<\varepsilon \big\}.
\]
We introduce also additional conditions.
\begin{itemize}
   \item[(W)] The Weierstrass condition holds if
      \begin{equation}
         E(t,u_0(t),\dot{u}_0(t), w) \geq 0,
         \label{eq:wcondition}
      \end{equation}
      for all $t \in [0,T]$ and for all $w \in \R^n$. 
   \item[(W')] The strengthened Weierstrass condition holds if 
      \begin{equation}
         E(t,y,v,w) \geq 0,
         \label{eq:swcondition}
      \end{equation}
      for $(t,y,v) \in RT(u_0,\varepsilon)$ and for all $w \in \R^n$.
\end{itemize}
\begin{remark}
   Note that condition (W') is satisfied whenever $L$ is globally convex in
   $\dot{u}$, that is when $L_{\dot{u}\dot{u}}\geq 0$. 
   In classical mechanics, the latter condition is fulfilled because the term of the
   kinetic energy which is quadratic in $\dot{u}$ is positive definite.
   \label{rmk:sw}
\end{remark} 
In the classical setting of the fixed endpoints problem, it is known that a necessary
condition for a solution of the Euler-Lagrange equation to be a (SLM) is
that condition (W) holds. This is true also in the case of periodic
boundary conditions: we can use the same proof, taking into account only the competitors that are
zero at the extrema of the time interval.
Moreover, in the fixed-end point problem it is known that a solution of the Euler-Lagrange equation satisfying conditions (L'), (J') and (W') is (SLM). We refer again to classical books of Calculus of Variations for proofs
of this fact \cite{giaquinta_2004, pars_1962}. A similar theorem is valid for periodic boundary conditions.

\begin{theorem}
    Let $u_0 \in X$ be a $C^2$ periodic solution of the Euler-Lagrange equation \eqref{eq:EL}. Suppose that conditions (L'), (W'), and the equivalent conditions of Theorem~\ref{th:DoslaSufCond} are satisfied for the second variation associated to $u_0$. Then $u_0$ is a (SLM).
    \label{th:slm_conj}
\end{theorem}
\begin{proof}
    Because (J') holds, then $\det Y_0(t)\neq 0$ in $(0,T]$. Therefore $u_0$ is a non-degenerate solution of the fixed-end point problem. Therefore, by the classical theory of Calculus of Variations, for $(y,z)$ in a neighborhood of $(u_0(0), u_0(T))$ there exists a unique solution of the Euler-Lagrange equation \eqref{eq:EL} $u_{(y,z)}$ which is close to $u_0$ in the $C^1$ norm. In addition, $u_{(y,z)}$ is $C^1$ in $(y,z)$. 
    This means that $u_{(y,z)}(t)$ is contained in a restricted tube $RT(u_0, \varepsilon)$ for some $\varepsilon>0$. By continuity therefore $u_{(y,z)}$ still satisfies (L'), (J'), and (W'), hence it is a (SLM) for the fixed-end point problem with boundary conditions $u(0)=u_{(y,z)}(0), u(T)=u_{(y,z)}(T)$.

    In a neighborhood of $y_0 = u_0(0) \in \R^n$ we define the reduced action
    \[
    f(y) = \A(u_{(y,y)}).
    \]
    Let us set $p(t)=L_{\dot{u}}(t,u,\dot{u})$, then $\nabla f(y) = p_y(T)-p_y(0)$, which is $0$ in $y_0$ because Eq.~\eqref{eq:Ludot} holds. In addition, because $u_{(y,y)}$ are all extrema, the Hessian is
    \[
    f''(y_0)[x,x] = \delta^2\A(v^*), \quad v^*(t) = (Y_0(t) Y_0(T)^{-1} + Y_T(t) Y_0(T)^{-T})x,
    \]
    which is exactly the Jacobi field with $v^*(0)=v^*(T)=x$. Therefore, with the same steps of \emph{(i)}$\Rightarrow$\emph{(iii)} of Theorem~\ref{th:DoslaSufCond}, we obtain
    \[
    f''(y_0) = W_0(T) - W_T(0) - Y_0(T)^{-1} - Y_0(T)^{-T},
    \]
    which is positive definite by hypothesis. Hence $y_0$ is a strict minimum point of $f$.

    Let us now take $u \in X$ such that $\norm{u-u_0}_\infty < \varepsilon$, with $\varepsilon$ such that $u(0)=u(T)=y$ is in the neighborhood of $u_0$, and such that $\norm{u - u_{(y,y)}}_\infty < \varepsilon_1$. Then $u$ is a competitor of $u_{(y,y)}$ for the fixed-end point problem, and therefore
    \[
    \A(u)\geq f(y)\geq f(y_0) = \A(u_0),
    \]
    hence $u_0$ is a (SLM).
\end{proof}

\begin{remark}
    In \cite{fenucci_2022} it was shown that conditions (L'), (W'), and (SR) are sufficient for a solution of the Euler-Lagrange equation to be a (SLM). However, Theorem~\ref{th:slm_conj} is a weaker statement, since (SR) is not necessary to satisfy its hypotheses.  
\end{remark}

Theorem~\ref{th:slm_conj} completes the characterization of local minimizers for periodic boundary conditions. These are all schematically summarized in Figure~\ref{fig:Minima_Scheme}.
\begin{figure}
    \centering
    \includegraphics[width=0.9\textwidth]{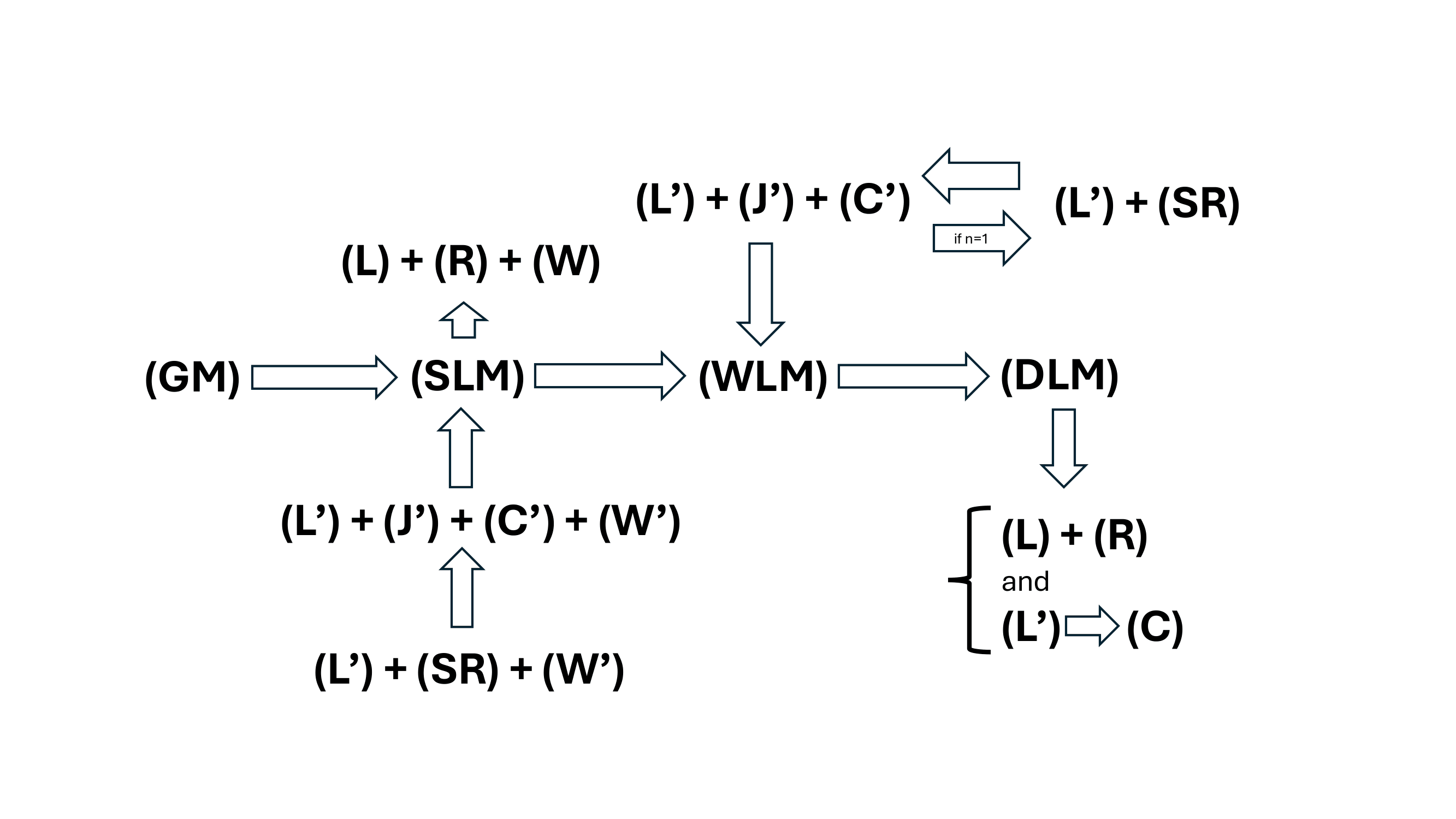}
    \caption{Summary scheme of the necessary and sufficient conditions for the local minimality under periodic boundary conditions.}
    \label{fig:Minima_Scheme}
\end{figure}

\section*{Declarations}
\noindent\textbf{Funding.} The author declares that no funds, grants, or other support were received during the preparation of this manuscript.

\noindent\textbf{Competing interests.} The author has no relevant financial or non-financial interests to disclose.

\noindent\textbf{Data availability.} Data sharing is not applicable to this article, as no datasets were generated or analysed during the current study.

\bibliographystyle{amsplain}
\bibliography{sample}

@book{Reid1980,
  title={Sturmian Theory for Ordinary Differential Equations},
  author={Reid, W.T.},
  year={1993},
  publisher={Springer-Verlag}
}

@article{barrabes-etal_2006,
   author={Barrab{\'e}s, E. and Cors, J. M. and Pinyol, C. and Soler, J.},
   title={Hip-hop solutions of the {$2N$}-body problem},
   journal={Celestial Mechanics and Dynamical Astronomy},
   year={2006},
   month={May},
   day={01},
   volume={95},
   number={1},
   pages={5--66},
}

@article{chen_2008,
 author = {Chen, K. C.},
 journal = {Annals of Mathematics},
 number = {2},
 pages = {325--348},
 title = {Existence and Minimizing Properties of Retrograde Orbits to
                  the Three-Body Problem with Various Choices of
                  Masses},
 volume = {167},
 year = {2008},
}

@article{chen_2012,
title = {Action-minimizing periodic and quasi-periodic solutions in
                  the {$N$}-body problem},
author = {Chen, K. C. and Tiancheng O. and Zhihong X.},
year = {2012},
volume = {19},
pages = {483--497},
journal = {Mathematical Research Letters},
publisher = {International Press of Boston, Inc.},
}

@article{hilscher-zeidan_2008,
  title={Applications of time scale symplectic systems without normality},
  author={Roman Simon Hilscher and Vera Zeidan},
  journal={Journal of Mathematical Analysis and Applications},
  year={2008},
  volume={340},
  pages={451-465},
  url={https://api.semanticscholar.org/CorpusID:120501287}
}

@article{chenciner-montgomery_2000,
 author = {Chenciner, A. and Montgomery, R.},
 journal = {Annals of Mathematics},
 number = {3},
 pages = {881--901},
 title = {A Remarkable Periodic Solution of the Three-Body Problem in
                  the Case of Equal Masses},
 volume = {152},
 year = {2000},
}

@article{fenucci_2022,
	author = {Fenucci, M. },
	date = {2022/01/18},
	doi = {10.1007/s42985-022-00148-5},
	isbn = {2662-2971},
	journal = {Partial Differential Equations and Applications},
	number = {1},
	pages = {10},
	title = {Local minimality properties of circular motions in {\$}{\$}1/r\^{}{$\backslash$}alpha {\$}{\$}potentials and of the figure-eight solution of the 3-body problem},
	url = {https://doi.org/10.1007/s42985-022-00148-5},
	volume = {3},
	year = {2022},
}

@book{dacorogna_2004,
  title={Introduction to the Calculus of Variations},
  author={Dacorogna, B.},
  year={2004},
  publisher={Imperial College Press}
}

@article{dosla-dosly_1997,
    AUTHOR = {Do\v{s}l\'a, Z. and Do\v{s}l\'y, O.},
     TITLE = {Quadratic functionals with general boundary conditions},
   JOURNAL = {Appl. Math. Optim.},
    VOLUME = {36},
      YEAR = {1997},
    NUMBER = {3},
     PAGES = {243--262},
}

@article{chenciner-venturelli_2000,
 author = {Chenciner, A. and Venturelli, A.},
 title = {Minima de l'int\'egrale d'action du probl\`eme newtonien de
                  4 corps de masses \'egales dans {${\bf R}^3$}:
                  orbites ``hip-hop''},
 journal = {Cel. Mech. Dyn. Ast.},
 volume = {77},
 number = {2},
 pages = {139--152},
 year = {2000},
}

@article{ferrario-terracini_2004,
author= {Ferrario, D. L. and Terracini, S.},
title={On the existence of collisionless equivariant minimizers for
                  the classical {$n$}-body problem},
journal={Inventiones mathematicae},
year={2004},
volume={155},
number={2},
pages={305--362},
}

@article{fenucci-gronchi_2018,
  author={Fenucci, M. and Gronchi, G. F.},
  title={On the stability of periodic {$N$}-body motions with the symmetry of {P}latonic polyhedra},
  journal={Nonlinearity},
  volume={31},
  number={11},
  pages={4935},
  year={2018},
}

@ARTICLE{fenucci-gronchi_2023,
       author = {{Fenucci}, Marco and {Gronchi}, Giovanni Federico},
        title = "{Symmetric Constellations of Satellites Moving Around a Central Body of Large Mass}",
      journal = {Journal of Dynamics and Differential Equations},
         year = 2023,
        month = jun,
       volume = {35},
       number = {2},
        pages = {1511-1559},
          doi = {10.1007/s10884-021-10083-5},
archivePrefix = {arXiv},
       eprint = {2003.04580},
 primaryClass = {math-ph},
       adsurl = {https://ui.adsabs.harvard.edu/abs/2023JDDE...35.1511F}
}

@article{fusco-gronchi-negrini_2011,
   author={Fusco, G. and Gronchi, G. F. and Negrini, P.},
   title={Platonic polyhedra, topological constraints and periodic
                  solutions of the classical {$N$}-body problem},
   journal={Inventiones mathematicae},
   year={2011},
   volume={185},
   number={2},
   pages={283--332},
}

@book{giaquinta_2004,
  title={Calculus of Variations {I}},
  author={Giaquinta, M. and Hildebrandt, S.},
  series={Grundlehren der mathematischen Wissenschaften},
  year={2004},
  publisher={Springer Berlin Heidelberg}
}

@book{galileo_1638,
  title={Discorsi e dimostrazioni matematiche intorno a due nuove scienze},
  author={Galilei, G.},
  year={1638},
  publisher={Lodewijk Elzevir}
}

@article {moore_1993,
    AUTHOR = {Moore, C.},
     TITLE = {Braids in classical dynamics},
   JOURNAL = {Phys. Rev. Lett.},
    VOLUME = {70},
      YEAR = {1993},
    NUMBER = {24},
     PAGES = {3675--3679},
}

@book{pars_1962,
  title={An introduction to the calculus of variations},
  author={Pars, L.A.},
  year={1962},
  publisher={Wiley}
}

@inbook{simo_2000,
   author={Sim\'o, C.},
   title={Periodic orbits of the planar {$N$}-body problem with equal
                  masses and all bodies on the same path},
   bookTitle={The restless universe. Applications of gravitational
                  $N$-body dynamics to planetary, stellar and galactic
                  systems},
   year={2001},
   publisher={IoP Publishing},
   pages={265--284},
}

@inproceedings{simo_2001,
   author={Sim{\'o}, C.},
   title={New Families of Solutions in {$N$}-Body Problems},
   editor={Casacuberta, Carles and Mir{\'o}-Roig, Rosa Maria and
                  Verdera, Joan and Xamb{\'o}-Descamps, Sebasti{\`a}},
   bookTitle={European Congress of Mathematics: Barcelona, July
                  10--14, 2000, Volume I},
   year={2001},
   publisher={Birkh{\"a}user Basel},
   address={Basel},
   pages={101--115},
}

@incollection{simo_2002,
    AUTHOR = {Sim\'o, C.},
     TITLE = {Dynamical properties of the figure eight solution of the
                  three-body problem},
 BOOKTITLE = {Celestial mechanics ({E}vanston, {IL}, 1999)},
    SERIES = {Contemp. Math.},
    VOLUME = {292},
     PAGES = {209--228},
 PUBLISHER = {Amer. Math. Soc., Providence, RI},
      YEAR = {2002},
}

@article{zeidan-zezza_1989,
    AUTHOR = {Zeidan, V. and Zezza, P.},
     TITLE = {Coupled points in the calculus of variations and applications
              to periodic problems},
   JOURNAL = {Trans. Amer. Math. Soc.},
    VOLUME = {315},
      YEAR = {1989},
    NUMBER = {1},
     PAGES = {323--335},
}

@article{zeidan_1993,
   author  = {Zeidan, V.},
   title   = {Sufficient conditions for variational problems with variable endpoints:
      Coupled points},
   journal = {Applied Mathematics and Optimization},
   year    = {1993},
   volume  = {27},
   number  = {2},
   pages   = {191--209},
   doi     = {10.1007/BF01195982},
}

@article{zeidan_2001,
    AUTHOR = {Zeidan, V.},
     TITLE = {New second-order optimality conditions for variational
              problems with {$C^2$}-{H}amiltonians},
   JOURNAL = {SIAM Journal on Control and Optimization},
    VOLUME = {40},
      YEAR = {2001},
    NUMBER = {2},
     PAGES = {577--609},
      ISSN = {0363-0129},
       DOI = {10.1137/S0363012999358725},
       URL = {https://doi.org/10.1137/S0363012999358725},
}

@ARTICLE{ciccarelli-etal_2026,
       author = {{Ciccarelli}, Roberto and {Introna}, Margaux and {Terracini}, Susanna and {Vasile}, Massimiliano},
        title = "{New Solutions for the Symmetrical n-Body Problem Through Variational Approach and Optimisation Techniques}",
      journal = {Aerotecnica Missili \& Spazio},
         year = 2026,
        month = apr,
       volume = {105},
       number = {2},
        pages = {153-170},
          doi = {10.1007/s42496-025-00265-5},
archivePrefix = {arXiv},
       eprint = {2505.04221},
 primaryClass = {math.DS},
       adsurl = {https://ui.adsabs.harvard.edu/abs/2026AeMiS.105..153C}
}

@ARTICLE{berti-etal_2026,
       author = {{Berti}, Diego and {Canneori}, Gian Marco and {Ciccarelli}, Roberto and {De Blasi}, Irene and {Introna}, Margaux and {Polimeni}, Davide},
        title = "{Symmetric solutions of the n-body problem: A numerical study of Floquet multipliers and Morse indices}",
      journal = {Astronomy and Computing},
         year = 2026,
        month = apr,
       volume = {55},
          eid = {101085},
        pages = {101085},
          doi = {10.1016/j.ascom.2026.101085},
archivePrefix = {arXiv},
       eprint = {2511.17234},
 primaryClass = {math.DS},
       adsurl = {https://ui.adsabs.harvard.edu/abs/2026A&C....5501085B}
}

@ARTICLE{fenucci-jorba_2020,
       author = {{Fenucci}, Marco and {Jorba}, {\`A}ngel},
        title = "{Braids with the symmetries of Platonic polyhedra in the Coulomb (N+1)-body problem}",
      journal = {Communications in Nonlinear Science and Numerical Simulations},
         year = 2020,
        month = apr,
       volume = {83},
          eid = {105105},
        pages = {105105},
          doi = {10.1016/j.cnsns.2019.105105},
archivePrefix = {arXiv},
       eprint = {2009.07901},
 primaryClass = {math-ph},
       adsurl = {https://ui.adsabs.harvard.edu/abs/2020CNSNS..8305105F}
}

\end{document}